\documentclass[11pt]{article}
\usepackage{currfile}
\usepackage{color}
\usepackage{mathrsfs}
\usepackage{graphicx}
\usepackage{caption}
\usepackage{subfigure}
\usepackage{amsmath}
\usepackage{amssymb}
\usepackage{setspace}
\usepackage{epstopdf}

\newtheorem{thm}{Theorem}[section]

\newtheorem{lem}[thm]{Lemma}
\newtheorem{prop}[thm]{Proposition}
\newtheorem{prob}[thm]{Problem}
\newtheorem{ques}[thm]{Question}
\newenvironment{proof}{\noindent {\bf
Proof.}}{\rule{3mm}{3mm}\par\medskip}

\def\marker{\>\hbox{${\vcenter{\vbox{
    \hrule height 0.4pt\hbox{\vrule width 0.4pt height 6pt
    \kern6pt\vrule width 0.4pt}\hrule height 0.4pt}}}$}\>}

\def\qed{ \hfill $\square$}
\newcommand{\w}{\chi'_{\rm{wo}}}

\begin{document}
\title{ Weak-odd chromatic index of special digraph classes }
\author{\small {Ruijuan Gu$^1$, Hui Lei$^2$\thanks{The corresponding author.}, Xiaopan Lian$^3$,  Zhenyu Taoqiu$^3$}\\
{\small $^1$ Sino-European Institute of Aviation Engineering}\\
{\small Civil Aviation University of China, Tianjin 300300, China}\\
{\small $^2$ School of Statistics and Data Science, LPMC and KLMDASR}\\
{\small Nankai University, Tianjin 300071, China}\\
{\small $^3$ Center for Combinatorics and LPMC}\\
{\small Nankai University, Tianjin 300071, China}\\
{\small Email: millet90@163.com; hlei@nankai.edu.cn;}\\ {\small xiaopanlian@mail.nankai.edu.cn; tochy@mail.nankai.edu.cn}}

\maketitle
\begin{abstract}
Give a digraph $D=(V(D),A(D))$, let $\partial^+_D(v)=\{vw|w\in N^+_D(v)\}$ and $\partial^-_D(v)=\{uv|u\in N^-_D(v)\}$ be semi-cuts of $v$. A mapping $\varphi:A(D)\rightarrow [k]$ is called a {\it weak-odd $k$-edge coloring} of $D$ if it satisfies the condition: for each $v\in V(D)$, there is at least one color  with an odd number of occurrences on each non-empty semi-cut of $v$. We call the minimum integer $k$ the {\it weak-odd chromatic index} of $D$.
When limit to 2 colors, use $def(D)$ to denote the {\it defect} of $D$, the minimum number of vertices in $D$ at which the above condition is not satisfied.
In this paper, we give a descriptive characterization about the weak-odd chromatic index and the defect of semicomplete digraphs and extended tournaments, which generalize results of tournaments to broader classes. And we initiated the study of weak-odd edge covering on digraphs.


\end{abstract}

\section{Introduction}
Throughout the paper, we follow the terminology and notion from \cite{BG,JU}. Here all digraphs considered are finite.


Give a graph $G=(V(G),E(G))$. Denote by $d_G(v)$ the number of edges incident with $v$ in $G$. 
A mapping $\varphi:E(G)\rightarrow [k]$ is called a {\it weak-odd k-edge coloring} of $G$ if it satisfies the following condition:
\begin{description}
	\item[(\rm WO)] For $v\in V(G)$ with $d_G(v)>0$, there is at least one color $i\in[k]$ such that the number of edges incident with $v$ colored by $i$ is odd.
\end{description}
We should note that this concept is a relaxation of odd edge coloring of graphs which was first introduced by Pyber in \cite{LP}. The {\it odd edge coloring} is an edge-coloring  such that at each non-isolated vertex
every appearing color is odd. The {\it weak-odd chromatic index} of $G$, denoted by $\w(G)$, is the minimum integer $k$ such that $G$ admits a weak-odd $k$-edge coloring. This concept motivated by \cite{JS,JSFR,LP} is given in \cite{MP}, where   Mirko gave an intuitive characterization of graphs in terms of their weak-odd chromatic index.


Inspired by the study of graphs, Petru\v{s}evski and \v{S}krekovski \cite{MR} generalized this concept to digraphs.
Given a digraph $D=(V(D),A(D))$, let $\partial^+_D(v)=\{vw|w\in N^+_D(v)\}$ and $\partial^-_D(v)=\{uv|u\in N^-_D(v)\}$ be {\it semi-cuts} of $v$. The {\it out-degree}(resp. {\it in-degree}) of $v$ which is also called the {\it semi-degree} of $v$, denoted by $d^+_D(v)$(resp. $d^-_D(v)$), is the cardinality of the set $\partial^+_D(v)$(resp. $\partial^-_D(v)$). We say a vertex $v\in V(D)$ is a {\it peripheral} vertex if either $d^+_D(v)=0$ or $d^-_D(v)=0$. Specifically, if $d^+_D(v)=0$, then $v$ is a {\it sink} of $D$, and if $d^-_D(v)=0$, then $v$ is {\it source}.
A mapping $\varphi:A(D)\rightarrow [k]$ is said to be a {\it weak-odd k-edge coloring} of $D$ if the following holds:
\begin{description}
	\item[($\overrightarrow{\rm{WO}}$)] For any $v\in V(D)$, there is at least one color $i\in[k]$ such that the number of arcs in each nonempty semi-cut of $v$ colored by $i$ is odd.
\end{description}
We say such $D$ {\it weak-odd $k$-edge colorable}, and call the suitable minimum integer $k$  {\it weak-odd chromatic index}, denoted by $\w(D)$.



In the same paper,
the authors showed that $\w(D)\le 3$ and the bound is sharp. They believed that a descriptive characterization similar to graphs is impossible for all digraphs and they believed that deciding the exact value of $\w(D)$ is NP-hard.
In \cite{CMR},
the authors showed a necessary and sufficient condition for digraphs to be weak-odd 2-edge colorable, and thus
$\w(D)$ can be determined in polynomial time. When limit to 2 colors, use ${\rm def}(D)$ to denote the {\it defect} of $D$, the minimum number of vertices in $D$ at which the condition ($\overrightarrow{\rm{WO}}$) is not satisfied. Hern\'{a}ndez-Cruz,  Petruevski, and  Krekovski \cite{CMR} proved that ${\rm def}(D)$ is related to the matching number of some graphs.

A {\it tournament} is an oriented graph where every pair of vertices are adjacent.
A digraph is called {\it semicomplete} if it is obtained from a complete graph by replacing each edge $(u,v)$ with the arc $uv$ or $vu$ or a pair of symmetric arcs. By {\it extended tournaments} we mean the digraph obtained from a tournament by blowing up some of its vertices into independent sets. Hern\'{a}ndez-Cruz,  Petruevski, and  Krekovski \cite{CMR} made a descriptive characterization for tournaments of the weak-odd chromatic index as follows.
\begin{thm}[\cite{CMR}]\label{wt=3}
For any tournament $T$, it holds that
$$\w (T)=\left\{
\begin{array}{rcl}
0 & &\text{if $T=K_1$,}\\
1& &\text{if $T$ is nontrivial and  every  vertex semi-degree is odd or zero},\\
3& &\text{if $T$ is nontrival, of odd order, and has just one peripheral vertex},\\
2& &\text{otherwise}.\\
\end{array}
\right.
$$
\end{thm}
And the defect of tournament is 1 when the case $\w(T)=3$.


Then they asked whether these results can be extended to generalization classes of tournaments.

\begin{prob}[\cite{CMR}]
Characterize the families of semicomplete digraphs, extended tournaments and multipartite tournaments in terms of their weak-odd chromatic index.
\end{prob}

\begin{prob}[\cite{CMR}]\label{woq}
Characterize the defect in terms of  the families of semicomplete digraphs, extended tournaments and multipartite tournaments  when their defect are bounded.
\end{prob}

We give the complete characterization about the above two problems for the first two graph classes, i.e., semicomplete digraphs and extended tournaments. The results can be helpful for the remaining class. And we think the result of multipartite tournaments is also optimistic.

%
%
%

Hern\'{a}ndez-Cruz,  Petruevski, and  Krekovski \cite{CMR} also started the study of weak-odd edge covering and showed the weak-odd 2-edge covering conditions for graphs. Then they asked the situation about digraphs.
For a digraph $D$, an {\it edge covering} with color set $S$ is a mapping that assigns to each arc of $D$ a nonempty subset of $S$. The {\it weak-odd edge covering} is defined as edge covering such that condition ($\overrightarrow{{\rm WO}}$) is satisfied.
%
\begin{ques}[\cite{CMR}]\label{cover}
	Does every digraph admit a weak-odd 2-edge covering?
\end{ques}

We give a positive answer about tournaments. This is of positive significance to the study of digraphs. We believe that similar research can be carried out on the simple generalization classes of tournaments.

\vspace{2mm}
The paper is organized as follows. In next section, we first introduce the notion and terminology that are not mentioned before, then we list some auxiliary tools
that will be used in our proof. Then we give descriptive characterizations of semicomplete digraphs and extended tournaments respectively in Sections 3 and 4.
In the last section, we prove that every tournament admits a weak-odd 2-edge covering.

\section{Preliminary}
Give a digraph $D=(V(D),A(D))$,  the {\it degree} of $v\in V(D)$ , denoted by $d_D(v)$, is the total number of arcs that incoming at $v$ and outgoing at $v$, thus $d_D(v)=d^+_D(v)+d^-_D(v)$. By saying a graph or a digraph is {\it even} we shall mean that each vertex in it has even degree. The {\it minimum out-degree} ({\it minimum in-degree}) of $D$ is $\delta^+(D)=\min\{d^+_D(v)|v\in V(D)\}$ ($\delta^-(D)=\min\{d^-_D(v)|v\in V(D)\}$). The {\it minimum semi-degree} of $D$ is $\delta^0(D)=\min \{\delta^-(D),\delta^+(D)\}$.
For $X,Y\subseteq V(D)$, let $A(X,Y)=\{uv\in A(D)|u\in X,v\in Y\}$. A directed $X$-$Y$ path is an $(x,y)$-dipath $P$ such that $V(P)\cap X=\{x\}$ and $V(P)\cap Y=\{y\}$.
The subdigraph of $D$ induced by $X\subseteq A(D)$ will be denoted by $G[X]$. A vertex $u$ is said to {\it dominate} a vertex $v$ if $v\in N^+_D(u)$.

A {\it strong component} of a digraph $D$ is a maximal induced subdigraph
of $D$ which is strong.  If $D_1,\ldots,D_t$ are the strong components of $D$, then
$V(D_i)\cap V(D_j)=\emptyset$ for every $i\neq j$
as otherwise all the vertices $V(D_i)\cup V(D_j)$ are reachable from each other. The {\it strong component digraph} $SC(D)$ of $D$ is obtained by contracting
the strong components of $D$ and deleting any parallel arcs obtained in this
process.
The strong components
of $D$ corresponding to the vertices of $SC(D)$ of in-degree (out-degree) zero
are the {\it initial (terminal)} strong components of $D$, which is also called the {\it peripheral strong component}.

We shall emphasize that when dealing with graphs, the conception {\it $S$-join} is a powerful tool. Given a graph $G=(V(G),E(G))$ and an even-sized vertex subset $S$, we call a spanning subgraph $H$ is an {\it $S$-join} of $G$ if $d_H(v)$ is odd for $v \in S$ while $d_H(v)$ is even for $v \in V(G)\setminus S$. And it has been proved that  {\it if $G$ is a connected graph, then $G$ contains an $S$-join for any even-sized vertex subset $S$} (see \cite{S2003}).
When turn attention to digraphs, the problem of  determining the weak-odd chromatic index of digraphs can be settled through constructing the following auxiliary graphs.
Given a digraph $D=(V,A)$, its {\it bipartite representation} or {\it split} is a bipartite graph $BG(D)=(V^+,V^-, E)$ where $V^+=\{v^+:~v\in V\}$, $V^-=\{v^-:~v\in V\}$, and $(u^+,v^-)\in E$ if and only if $uv\in A$. The {\it partial split}, $PS(D)$, of $D$ is a graph obtained from $BG(D)$ by re-identifying each pair $(u^+; u^-)$
for which both $d^+_D(u)$ and $d^-_D(u)$ are odd. See Figure~\ref{wod1}.
\begin{figure}[htbp]
  \centering
  \includegraphics[width=14cm]{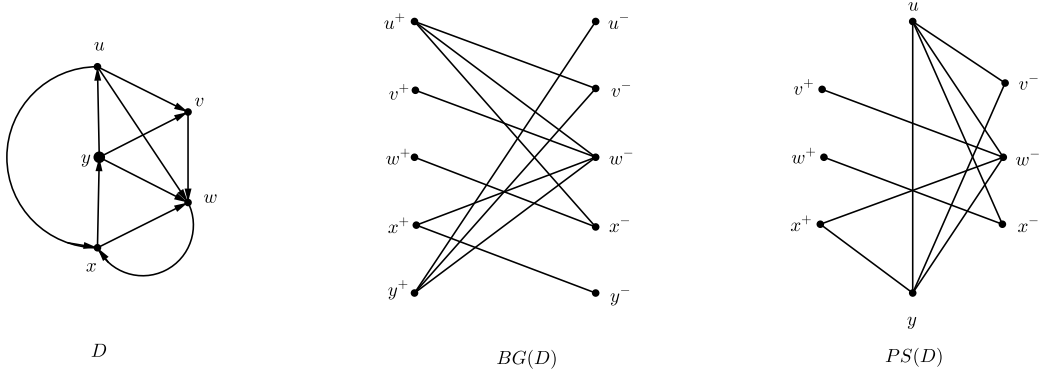}\\
  \caption{The split graph $BG(D)$ and partial split graph $PS(D)$ of $D$}\label{wod1}
\end{figure}

To solve the problem that whether a digraph is weak-odd 2-edge colorable, Hern\'{a}ndez-Cruz,  Petruevski, and  Krekovski \cite{CMR} defined a 3-partition $\{V_1; V_2;V_3\}$ of $V (PS(D))$:
\begin{itemize}
\item[$\bullet$]$V_1= V (D)\cap V (PS(D))$, i.e., $V_1$ is consisted of the vertices $u$ of $D$ with both $d^+_D(u)$ and $d^-_D(u)$ are odd.
\item[$\bullet$]$V_2=\{v\in V(PS(D))\setminus V_1: d_{PS(D)}(v) \text{ is even}\}$.
\item[$\bullet$] $V_3=\{v\in V(PS(D))\setminus V_1: d_{PS(D)}(v) \text{ is odd}\}$.

\end{itemize}
 We say a component $K$ of $PS(D)$ is `bad' if $V(K)\cap V_2$ is of odd size and $V(K)\cap V_3=\emptyset$.
Define a graph $G_D=(V_D,E_D)$, the vertex set consists of vertices $v_K$ corresponding to bad components $K$  and  two distinct vertices $v_{K'}$ and $v_{K''}$ are adjacent if
the respective bad components $K'$ and $K''$ contain the `halves' $v^+$
and $v^-$ of some vertex  $v\in V(D)$. Let $\alpha'_D$ be the cardinality of the maximum matching of $G_D$.
The following results proved in \cite{CMR} will be used later.
\begin{thm}[\cite{CMR}]\label{dkv23}
A digraph $D$ is weak-odd 2-edge colorable if and only if for every nontrivial component $K$ of $PS(D)$ we have that $V(K) \cap V_2$ is even-sized or $V(K)\cap V_3\neq \emptyset$.
\end{thm}

\begin{prop}[\cite{CMR}]\label{wd=3}
If an even digraph $D$ has an odd number of peripheral vertices, then
$\w(D)=3$.
\end{prop}


\begin{thm}[\cite{CMR}]\label{DG-A}
For every digraph $D$, ${\rm def}(D)=n(G_D)-\alpha'_D$ holds.
\end{thm}
%

Finally, we give a useful statement  about the weak-odd edge coloring. Let $D=(V(D),A(D))$ be a digraph and $v\in V(D)$. Let $D'$ be the digraph obtained from $D$ by deleting $v$. If $D'$ admits a weak-odd 2-edge coloring $\phi$, then we define a 2-edge coloring $\varphi$ of $D$ such that ($\overrightarrow{{\rm WO}}$) is satisfied for each vertex apart from $v$ as follows.

\begin{description}
	\item[(C)] For each $u\in N_{D}(v)$, suppose that color $i$ satisfies the condition ($\overrightarrow{{\rm WO}}$) at $u$ for $\phi$, where $i\in[2]$. If $uv\in A(D)$,  then coloring $uv$ with color $3-i$ when $d^+_{D'}(u)>0$ and color $i$ when $d^+_{D'}(u)=0$. If $vu\in A(D)$, then coloring  $vu$ with color $3-i$ when $d^-_{D'}(u)>0$ and color $i$ when $d^-_{D'}(u)=0$.
	
\end{description}

\section{Semicomplete digraphs}
We first state some simple properties of semicomplete digraphs, which can be find in Section
2 of \cite{BG}:
(i) every semicomplete digraph has a hamiltonian dipath; (ii) every nontrivial strong semicomplete digraph contains a hamiltonian dicycle; (iii) the strong component digraph of an semicomplete digraph is an acyclic tournament and has an acyclic ordering of vertices; (iv) every semicomplete digraph has only one initial (terminal) strong component.

For simplicity of presentation, we call every nontrivial even semicomplete digraph having only one peripheral vertex bad and others good in the following.
\begin{thm}\label{wsd3}
For any semicomplete digraph $D$, it holds that
$$\w (D)=\left\{
\begin{array}{rcl}
0 & &\text{if $D=K_1$},\\
1& &\text{if $D$ is nontrivial and every vertex semi-degree is odd or zero,}\\
3& &\text{if $D$ is a nontrival even digraph with just one peripheral vertex},\\
2& &\text{otherwise}.\\
\end{array}
\right.
$$
\end{thm}

\begin{proof}
By Proposition~\ref{wd=3} and $\w (D)\le 3$,  it suffices to show that every good semi-complete digraph is weak-odd 2-edge colorable.
	

Give a good semicomplete digraph $D$. In the following, we always first  find a spanning subdigraph $\hat{D}$ of $D$. Then we define a 2-edge coloring $\theta$
of $D$ as  the arc set of $\hat{D}$ with color 1 and $A(D)-A(\hat{D})$ with color 2. It is easy to check ($\overrightarrow{{\rm WO}}$) holds for every vertex of $D$ under $\theta$ in each case.

If $D$ is strong, then let $\hat{D}$ be a Hamilton dicycle. If $D$ has two trivial peripheral strong components, say $x,y$, then let $\hat{D}$ be a $(x,y)$-Hamilton dipath.
If both peripheral strong components of $D$ are nontrivial, then there exists a directed $K_i$-$K_j$ path $P$ in $D$ that passes through every vertex $v\notin V(K_i)\cup V(K_j)$, where $K_i$ and $K_j$ are the initial and terminal strong components of $D$. Let $C_i$ and $C_j$, respectively, be hamiltonian
dicycles in $K_i$ and $K_j$. Denote by $x$ and $y$, respectively, the initial and terminal
vertex of $P$. We have $xy\notin A(P)$ if $P$  is of length $\ell(P)>1$.  Let $\hat{D}=D[A(C_i\cup C_j)]$ when $\ell(P)=1$ and
$\hat{D}=D[A(C_i\cup C_j\cup P)\cup \{xy\}]$ when $\ell(P)>1$.
Then color 1 meets condition ($\overrightarrow{\rm{WO}}$) for above cases.

We complete the proof by supposing that exactly one peripheral strong component of $D$, without loss of generality, the terminal one, is trivial, denoted by $\{y\}$. Then $y$ is the sink of $D$. Now there is a vertex $v\in V(D)$ such that $d_D(v)$ is odd as $D$ is good. Let $D'$ be the semicomplete digraph obtained from $D$ by deleting the vertex $v$ ($v=y$ if $d_D(y)$ is odd). We proceed by distinguishing whether $v=y$.

{\bf Case 1.} $v=y$.

First suppose that $D'$ does not contain peripheral strong components.
Then we have $\w(D')\le2$ by the above analysis. Let $\phi$ be a weak-odd 2-edge coloring of $D'$ and $\varphi$ be a 2-edge coloring defined as in {\bf (C)}.  Since $d_D(v)$ is odd,  color 1 or 2 satisfies the condition ($\overrightarrow{{\rm WO}}$) at $v$ under $\varphi$. Hence, $\varphi$ is a weak-odd 2-edge coloring of $D$.

Now we may assume that there exists a sink in $D'$, say $y'$.
Let $K$ be the initial strong component of $D$, and $C$ be a hamiltonian dicycle in $K$. Take a directed $K$-$y'$ path $P$ in $D'$ that pass through every vertex not in $V(K)$. Let $x$ be the initial vertex of $P$. Then we let $\hat{D}=D[A(C\cup P)\cup \{xv, y'v\}]$. Then $\theta$ is a weak-odd 2-edge coloring of $D$ because $d_D(v)$ is odd.

{\bf Case 2.} $v\neq y$.

Directly $y$ is still the sink of $D'$ and $d^+_D(v),d^-_D(v)>0$.
First suppose that $D'$ has another peripheral vertex, say $x$. Then $x$ is the source of $D'$.
Obviously, $vx$ and $vy$ are contained in $A(D)$. Let $P$ be a hamiltonian dipath in $D'$.
If $d^+_D(v)$ is odd, then there is a vertex $w\in V(P)$ such that $wv\in A(D)$ and $wy\notin A(P)$, and let $\hat{D}=D[A(P)\cup\{vx,wv,wy,vy\}]$. Otherwise, let $\hat{D}=D[A(P)\cup \{vx\}]$.

Now $D'$ has exactly one peripheral vertex $y$. Suppose that $V(D')=V(K)\cup \{y\}$ where $K$ is the initial strong component of $D'$.
Let $C$ be a hamiltonian dicycle in $K$.
If $d^+_D(v)$ is odd, then there is a vertex $w\in V(C)$ such that $wv,wy\in A(D)$. Let $\hat{D}=D[A(C)\cup\{wy,wv\}]$. Otherwise, let $\hat{D}=D[A(C)\cup \{vy\}]$.


Finally, we consider the case that $V(D')\neq V(K)\cup \{y\}$. Take a directed $K$-$y$ path $P$ in $D'$ that passes through every vertex not in $V(K)$. Let $x$ be the initial vertex of $P$. By our latest assumption, the arc $xy\notin A(P)$.
If $d^+_D(v)$ is even, then let $\hat{D}=D[A(C\cup P)\cup\{vy,xy\}]$.
Now assume that $d^+_D(v)$ is odd. If $xv\notin A(D)$, then there is a vertex $w\in V(D')\setminus \{x,y\}$ such that $wv\in A(D)$ and $wy\notin A(P)$. Let $\hat{D}=D[A(C\cup P)\cup\{wy,wv,xy\}]$.
Otherwise, let $\hat{D}=D[A(C\cup P)\cup \{xv\}$.


Then $\theta$ is a weak-odd 2-edge coloring of $D$ for Case 2. Indeed, color 2 fits the condition ($\overrightarrow{{\rm WO}}$) at $v$ while color 1 works for every other vertex.
\end{proof}

\begin{prop}\label{propsdd}
For any semicomplete digraph $D$, it holds that
$${\rm def}(D)=\left\{
\begin{array}{rcl}
1 & &\text{if $D$ is a nontrival even digraph with just one peripheral vertex},\\
0& &\text{otherwise}.\\
\end{array}
\right.
$$
\end{prop}

\begin{proof}
By Theorem~\ref{wsd3}, we may assume that $D$ is bad and has a sink $y$. Let $D'$ be the digraph obtained from $D$ by deleting the vertex $y$. It is not hard to find that $D'$ is not an even semicomplete digraph, thus $\w(D')\le2$. Apply to $A(D)$ the particular 2-edge coloring constructed as {\bf (C)}. The condition ($\overrightarrow{{\rm WO}}$) is satisfied at each vertex apart from $y$.
\end{proof}

\begin{prop}
Every bad semicomplete digraph $D$ admits a 2-edge coloring such that condition ($\overrightarrow{{\rm WO}}$) is satisfied at each vertex apart from a prescribed vertex $v\in V(D)$.
\end{prop}
\begin{proof}
We may assume that $D$ has a sink $y$. If $v=y$, then by Proposition~\ref{propsdd}, we are done. Suppose that $v\neq y$.
Note that $PS(D)$ has only one nontrivial component $K$. Observe that $V(K)\cap V_2= V_2\setminus \{y^+\}$ is odd-sized, and $V_3=\emptyset$.
If $v\in V_1$, then let $S=\{v\}\cup (V_2\setminus\{y^+\})$. If $v^+\in V_2$, then let $S=V_2\setminus\{v^+,y^+\}$. Take an $S$-join $H$ in $K$, and then color $E(H)$ with color 1 and the rest edges of $K$ with color 2. The inherited 2-coloring of $D$ fits the condition.
\end{proof}

\section{Extended tournaments}
In this section, we characterize the weak-odd chromatic index of extended tournaments. Let $D=(V,A)$ be a digraph with $V=\{v_1,\ldots,v_n\}$.
Blow up $v_1,\ldots, v_n$ into independent sets $I_1,\ldots,I_n$ of size $s_1, \ldots, s_n$ respectively, where $|s_i|\ge1,i\in[n]$. We call the resulted digraph an extended digraph of $D$ and denote it by $ED$. Without loss of generality, suppose that $s_1,\ldots,s_\ell$ are odd and others are even where $\ell \le n$.
Denote by $v^1_i,\ldots,v^{s_i-1}_i$ the other $s_i-1$ copies of $v_i$ in $ED$ for $i\in [n]$. Let $I^+_i=\{v^+, v^{1+}_i,\ldots,v^{s_{i-1}+}_i\}$ and $I^-_i=\{v^-, v^{1-}_i,\ldots,v^{s_{i-1}-}_i\}$ for $i\in [n]$. Let $V'_1,V'_2,V'_3$ and $V_1,V_2,V_3$ be the vertex partitions of $PS(D)$ and $PS(ED)$ as defined in section 2, respectively.


%

\begin{thm}\label{sodd}
If $\ell=n$, then $\w(ED)=\w(D)$.
\end{thm}

\begin{proof}
For any vertex $u\in V(ED)$, we have $d^+_{ED}(u)\equiv d^+_D(u)\pmod2$, $d^-_{ED}(u)\equiv d^-_D(u)\pmod2$ since each $s_i$ is odd. Therefore we have  $V'_i\subseteq V_i$ and $|V_i|\equiv |V'_i|\pmod2$.
%
Thus $|V(K)\cap V_i|\equiv |V(K')\cap V'_i|\pmod2$  and $|V(K)\cap V_i|=0$ if and only if $|V(K')\cap V'_i|=0$ for $i=2,3$.
By Theorem~\ref{dkv23}, $\w(ED)=\w(D)$.
\end{proof}

In the following, let $D$ be a tournament $T$, $Q_1=\bigcup_{i=1}^\ell I_i$ and $Q_2=\bigcup_{i=\ell+1}^n I_i$. Then we have $V(ET)=Q_1\cup Q_2$. Denote the order of $ET$, $Q_1$ and $Q_2$ by $q$, $q_1$ and $q_2$, respectively. Obviously, we have $q=q_1+q_2$ and $d_{ET}(u)=q-s_i$ for any $u\in I_i$.


\begin{lem}\label{V2}
Let $V_2$ be the vertex set of $PS(ET)$ as defined before, then the cardinality of $V_2$ is always even.
\end{lem}

\begin{proof}
First suppose that $q$ is even. If $u\in Q_1$, then $u$ attributes 1 to $|V_2|$ as $d_{ET}(u)$ is odd. Otherwise, either $u\in V_1$ or $u$ attributes 2 to $|V_2|$. Therefore $|V_2|\equiv q_1\pmod2$ is even as $q_1=q-q_2$ is even.
Now suppose that $q$ is odd. If $u\in Q_1$, then either $u\in V_1$ or $u$ attributes 2 to $|V_2|$ as $d_{ET}(u)$ is even. Otherwise, $u$ attributes 1 to $|V_2|$. Therefore $|V_2|\equiv q_2\pmod2$ is even as $q_2$ is even.
\end{proof}

\begin{thm}\label{ETle2}
	If $|V(T)|\le 3$, then $\w(ET)\le2$.
\end{thm}
\begin{proof}
If $T=K_1$, then $\w(ET)=0$. Suppose that $T=K_2=v_1v_2$. If both $v_1$ and $v_2$ are in $Q_1$, then $\w(ET)=1$. Otherwise, the only nontrivial component $PS(ET)$  satisfies  Theorem~\ref{dkv23}, then $\w(ET)\le2$.

Now suppose that $|V(T)|=3$. Suppose $\ell=3$, i.e., $s_1$, $s_2$ and $s_3$ are all odd. Then  $\w(ET)=\w(T)=1$ by Theorem~\ref{sodd}. It suffices to consider the following two cases under $\ell\leq 2$.

{\bf Case 1.} $T$ is a dicycle and $A(T)=\{v_1v_2,v_2v_3,v_3v_1\}$.

 Suppose $\ell=2$.
Then $PS(ET)$ has two nontrivial components $K$ and $R$ with $V(K)=I^+_1\cup I^-_2$ and $V(R)=I^-_1\cup I^+_2\cup I_3$.  Suppose $\ell\leq 1$. Then $PS(ET)$ has three nontrivial components $K, R, S$ such that $V(K)=I^+_1\cup I^-_2$, $V(R)=I^+_2\cup I^-_3$, $V(S)=I^+_3\cup I^-_1$. If $\ell=2$, then we have $V(K)\cap V_3\neq\emptyset$ and $V(R)\cap V_2$ is of even size. If $\ell=1$, then we have $V(R)\cap V_2$ is of even size and $V(F)\cap V_3\neq\emptyset$ for $F\in \{K, S\}$.
If $\ell=0$, then we have $V(F)\cap V_2$ is of even size for $F\in \{K, R, S\}$. Hence, by Theorem~\ref{dkv23}, $\w(ET)\le 2$.

{\bf Case 2.} $T$ has two peripheral vertices.

Let $v_i$ and $v_j$ be the source and the sink of $T$, respectively. Set  $\{i,j,t\}=[3]$. Then $PS(ET)$ has one nontrivial component $K$ such that $V(K)=I^+_i\cup I^-_j\cup I_t$ when $s_i$ and $s_j$ are odd, and $V(K)=I^+_i\cup I^-_j\cup I^+_t\cup I^-_t$ otherwise.
If $s_i+s_j$ is even, then $V(K)\cap V_2$ is of even size. Otherwise,
$V(K)\cap V_3\neq\emptyset$. Hence, by Theorem~\ref{dkv23}, $\w(ET)\le2$.
\end{proof}

In the following, we consider the case that $|V(T)|>3$. An extended tournament $ET$ with $|V(T)|>3$ is called \emph {bad} if all the following conditions are satisfied.
\begin{itemize}
\item[(a)] $ET$ is of odd order;
\item[(b)] Exactly one independent set is even, i.e., $|I_n|$ is even;
\item[(c)] $N^+_{T}(v_n)$ dominate $N^-_{T}(v_n)$;
\item[(d)] Either $|N^-_T(v_n)|=1$ or $|N^+_T(v_n)|=1$.
\end{itemize}

\noindent We call every other extended tournament \emph {good}.

\begin{thm}\label{BTB3O2}
For $n>3$ and $\ell<n$, we have $\w(ET)=3$ when $ET$ is bad and $\w(ET)\le2$ otherwise.
\end{thm}

\begin{proof}
Recall that $|V(ET)|=q$. First consider the case that $ET$ is bad. Without loss of generality, let $|N^+_T(v_n)|=1$ and $N^+_{T}(v_n)=v_i$.
Then $d^-_{ET}(v_i)=s_n$ and $d^+_{ET}(v_i)=d^-_{ET}(v_n)=q-s_n-s_i$ are even, and $d^+_{ET}(v_n)=s_i$ is odd. And for $u\in N^-_{ET}(v_n)$, $d_{ET}(u)$ is even, we have either $u\in V_1$ or $u$ attributes 2 to $|V_2|$.
Observe that $PS(ET)$ contains exactly two components $K$ and $R$ with $V(K)=I^+_n\cup I^-_i$ and $V(R)=V(PS(ET))\setminus V(K)$. Note that $I^-_i\subseteq V_2$, $|I^-_i|$ is odd and $V_3=I^+_n$. So, we have $|V(R)\cap V_2|$ is odd by Lemma \ref{V2}  and $V(R)\cap V_3=\emptyset$ as $V_3\subseteq V(K)$. Therefore, by Theorem~\ref{dkv23}, $\w(ET)=3$.

Now consider the case that $ET$ is good. We proceed our proof by considering the number of peripheral vertices in $T$.

First suppose that $T$ has a source $v_i$ and a sink $v_j$. Then $PS(ET)$  consists exactly one nontrivial component $K$ with $V(K)=V(PS(ET))\setminus (I^-_i\cup I^+_j)$.
If $v_i,v_j\in Q_1$, then $V(K)\cap V_3\neq\emptyset$ when $q$ is even and $V(K)\cap V_2$ is of even order or $V(K)\cap V_3\neq\emptyset$ when $q$ is odd.
Consider without loss of generality that $v_i\in Q_1$ and $v_j\in Q_2$. If $q$ is even, then $I^+_i\subseteq V_3$, otherwise $I^-_j\subseteq V_3$.
Now $v_i,v_j\in Q_2$. By Lemma \ref{V2}, we have $V(K)\cap V_2$ is of even order. Therefore, by Theorem~\ref{dkv23}, $\w(ET)\le2$.

Next suppose that $T$ has a peripheral vertex $v_j$, without loss of generality, we say $v_j$ is a sink, then we have $V(PS(ET))=V(K)\cup I^+_j$, where $K$ is a nontrivial component of $PS(ET)$. Assume that $v_j\in Q_1$. If $q$ is odd, then $\emptyset\neq V_3\subseteq V(K)$ as $Q_2\neq\emptyset$.
Otherwise, $V(K)\cap V_3\neq\emptyset$ as $I^-_j\subseteq (V_3\cap V(K))$. If $v_j\in Q_2$, then $K$ is satisfied to Theorem~\ref{dkv23} by Lemma \ref{V2}. Therefore, $\w(ET)\le2$.

Finally suppose that $\delta^0(T)\ge1$. We choose a vertex $v_i$ from $Q_2$ such that $A(N^-_T(v_i),$ $N^+_T(v_i))\neq\emptyset$, otherwise, let $v_i$ be any vertex in $Q_2$. Now we present a vertex partition $X\cup U\cup W$ of $PS(ET)$ with respect to $v_i$. Let $X=I_i$ if $v_i\in V_1$, otherwise, let $X=X_1\cup X_2$ with $X_1=I^+_i$ and $X_2=I^-_i$. Define $U=U_1\cup U_2\cup U_3$ and $W=W_1\cup W_2\cup W_3$ as follows.
$$
\begin{aligned}
	U_1&=\{u^+\colon u\in N^-_{ET}(v_i)\setminus V_1\},~
	U_2=\{u^-\colon u\in N^-_{ET}(v_i)\setminus V_1 \},~
	U_3=N^-_{ET}(v_i)\cap V_1;\\
	W_1&=\{w^+\colon w\in N^+_{ET}(v_i)\setminus V_1\},~
	W_2=\{w^-\colon w\in N^+_{ET}(v_i)\setminus V_1\},~
	W_3=N^+_{ET}(v_i)\cap V_1;\\	
\end{aligned}
$$

Suppose $PS(ET)$ is a connected graph. Then  we have $V_2$ is of even order or $V_3$ is nonempty by Lemma \ref{V2}. Hence, by Theorem~\ref{dkv23}, $\w(ET)\le2$.
So we consider  $PS(ET)$ is not a connected graph in the following.
If both $U_1$ and $W_1$ are empty sets, then $U_3$ and $W_3$ are nonempty. Since $T$ is a tournament, there are edges between $U_3$ and $W_3$. Thus $PS(ET)$ is a connected graph.
Therefore, without loss of generality, we may assume that $U_1\neq \emptyset$. Let $K$ be the nontrivial component of $PS(ET)$ that contains $U_1$. It suffices to show the following two claims.

{\bf Claim 1} If  $v_i$ satisfies $A(N^-_T(v_i),N^+_T(v_i))\neq\emptyset$, then $\w(ET)\le2$.

\noindent {\bf Proof.} We have $U_1\cup U_3\cup X\cup W_3\cup W_2\subseteq V(K)$.
If $U_2\subseteq V(K)$, then $PS(ET)$ is connected because  $\delta ^0(ET)\ge1$ and $W_1$ (if exists) is not an independent set of $PS(ET)$. So we
assume that there exists a vertex $v^-_j\in U_2$ such that $v^-_j\notin V(K)$, then $v_j$ is a source of the subdigraph of $ET$ induced by $N^-_{ET}(v_i)$.  Since $d^-_T(v_j)>0$, there must be a vertex $v^+_t\in W_1$ such that $v^+_tv^-_j\in E(PS(ET))$ and $v^+_t\notin V(K)$. We have $v_t$ is a sink of subdigraph of $ET$ induced by $N^+_{ET}(v_i)$. Thus $PS(ET)$ has two nontrivial components $K$ and $R$ with $R=I^-_j\cup I^+_t$.
If both $s_j$ and $s_t$ are even, then $V(R)\cap V_2$ is of even size and  so $V(K)\cap V_2$ is of even size by Lemma~\ref{V2}.
If both $s_j$ and $s_t$ are odd, then $V(R)\cap V_3\neq\empty$ and $V(K)\cap V_2$ is of even size by Lemma~\ref{V2}.  If exactly  one of $s_i$ and $s_j$ is odd, then $V(R)\cap V_3\neq\emptyset$, and $V(K)\cap V_3\neq\emptyset$ because $d_{ET}(v_i)=q-s_i$, $s_i$ is even and $d_{ET}(v^+_j)=d_{ET}(v^-_t)=q-s_j-s_t$.
Hence, by Theorem~\ref{dkv23}, $\w(ET)\le2$.\qed

{\bf Claim 2} If $A(N^-_T(v),N^+_T(v))=\emptyset$ for each $v\in Q_2$, then $\w(ET)\le2$.\\
\noindent {\bf Proof.}
If $|N^+_T(v_i)|\ge2$ and $|N^-_T(v_i)|\ge2$, then  $PS(ET)$ is a connected graph.
So without loss of generality, we assume $|N^-_T(v_i)|\ge2$ and $|N^+_T(v_i)|=1$. If $v_i\in V_1$ or  $W_3\neq\emptyset$, then  $PS(ET)$ is a connected graph. So we assume
$v_i\notin V_1$ and $W_3=\emptyset$ in the following.
Let $v_j\in N^+_{ET}(v_i)$.
Thus $PS(ET)$ has two nontrivial components  $K$ and $R$ with $R=I^+_i\cup I^-_j$.
If $s_j$ is even, then we have  $V(R)\cap V_2$ is of even size and so $V(K)\cap V_2$  is of even size by Lemma \ref{V2}.
If $s_j$ is odd, then $q$ is odd because $v_i\notin V_1$.
Since $ET$ is good, there is a vertex $u\in Q_2$ with $u\notin I_i$.
Then $V(R)\cap V_3=I^+_i$ and $V(K)\cap V_3\neq \emptyset$ as $d_{ET}(u)$ is odd and $\{u^+,u^-\}\subseteq V(K)$. Again by Theorem~\ref{dkv23}, $\w(ET)\le2$.
  
This completes the proof of Theorem \ref{BTB3O2}.
\end{proof}

\begin{thm}\label{ET=123}
For any extended tournament $ET$, it holds that
$$\w(ET)=\left\{
\begin{array}{rcl}
	0 & &\text{if $T=K_1$},\\
	1& &\text{if $\ell=n$ and $\w(T)=1$,}\\
	3 & &\text{if $\ell=n$ and $\w(T)=3$ or $ET$ is bad},\\
	2& &\text{otherwise}.\\
\end{array}
\right.
$$
\end{thm}

\begin{proof}
By Theorem~\ref{sodd}, Theorem~\ref{ETle2} and Theorem~\ref{BTB3O2}, we only need to declare the case that $\w(ET)=1$. Clearly, if there are $s_i$ and $s_j$ with different parity, then $\w(ET)\ge2$.  By Theorem~\ref{sodd}, if each $s_i$ is odd, then $\w(ET)=\w(T)$. Therefore, in this case, $\w(ET)=1$ if and only if $\w(T)=1$. Now assume that each $s_i$ is even. We have for each vertex $v\in ET$, both $d^+_{ET}(v)$ and $d^-_{ET}(v)$ are even, then $\w(ET)\ge2$.
\end{proof}

\begin{prop}\label{expre}
If $\w(ET)=3$,  then $ET$ admits a 2-edge coloring such that condition ($\overrightarrow{{\rm WO}}$) is satisfied at each vertex apart from a prescribed vertex $x\in V(T)$.
\end{prop}

\begin{proof}
If $\ell=n$ and $\w(T)=3$, then we have $|V(T)|$ is odd and $T$ has a peripheral vertex, say a sink $v_i$, by Theorem~\ref{wt=3}. Hence, $q=|V(ET)|$ is odd. 
Observe that  $PS(ET)$ has exactly one nontrivial component $K$ with $V(PS(ET))=V(K)\cup I^-_i$. If $x\in V_1$, then let $F=\{x\}\cup (V_2\setminus I^-_i)$. If $x^+\in V_2$, then let $F=V_2\setminus (\{x^+\}\cup I^-_i)$. Take an $F$-join $H$ in $K$, and color $E(H)$ with color 1 and the rest edges of $K$ with color 2. The inherited 2-coloring fits the condition.

Suppose now $ET$ is bad. We have $|V(ET)|$ is odd and $\delta^0(T)\ge1$ by the definition. Then $\w(T)\le 2$ by Theorem~\ref{wt=3}. Let $ET'$ be the extended tournament obtained from $ET$ by deleting $x$. If $\w(ET')\le 2$, then we obtain a 2-edge coloring of $ET$ from $ET'$ by the coloring {\bf(C)} such that the condition ($\overrightarrow{{\rm WO}}$) is satisfied at each vertex apart from $x\in V(ET)$. Now we only need to show that $\w(ET')\le2$.
If $x\in Q_2$, then each independent set is odd in $ET'$. So $\w(ET')=\w(T)\le2$. If $x\in Q_1$, then we have $ET'$ is good. Therefore, by Theorem~\ref{ET=123}, we have $\w(ET')\le2$.
\end{proof}
By Proposition \ref{expre}, we have the following proposition directly.
\begin{prop}\label{propebt}
For any extended tournament $ET$, it holds that
$${\rm def}(ET)=\left\{
\begin{array}{rcl}
1 & &\text{if $\ell=n$ and $\w(T)=3$ or $ET$ is bad},\\
0& &\text{otherwise}.\\
\end{array}
\right.
$$
\end{prop}

\section{Weak-odd edge covering of tournaments}
Here we show that Question~\ref{cover} holds for tournaments.

\begin{prop}\label{wc2}\cite{CMR}
	Every tournament admits a 2-edge coloring such that condition ($\overrightarrow{{\rm WO}}$) is satisfied at each vertex apart from a prescribed vertex $v\in V(T)$.
\end{prop}

\begin{thm}
Let $T$ be a tournament. If $\w(T)=3$, then $T$ admits a weak-odd 2-edge covering $\varphi$ such that the intersection of color classes is contained within a singleton arc in $A(T)$.
\end{thm}

\begin{proof}
By Theorem~\ref{wt=3}, we have that $T$ is nontrivial, of odd order, and has just one peripheral vertex. Suppose that $T$ has a sink $y$. By Proposition~\ref{wc2}, $T$ admits a 2-edge coloring $\phi$ such that ($\overrightarrow{{\rm WO}}$) is satisfied at each vertex apart from $y$. Let $T'_i$ be the spanning subdigraph of $T$ whose arc set is the color set $\phi^{-1}(i)$ for $i\in [2]$. Then
both $d^-_{T'_1}(y)$ and $d^-_{T'_2}(y)$ are even.
If there is an arc $xy\in A(T)$ such that $\phi(xy)=i$ and ($\overrightarrow{{\rm WO}}$) is satisfied by color $i$ at $x$ for $i\in [2]$, then $\varphi$ can be given as follows: $\varphi(xy)=\{1,2\}$ and $\varphi=\phi$ for other arcs.
It is obvious to see that the condition ($\overrightarrow{{\rm WO}}$) is satisfied by color  $3-i$ at $y$ and others are taken care of by the same color  under $\phi$ for $i\in [2]$.

Now suppose that for any arc $xy\in A(T)$ and $i\in [2]$, if $\phi(xy)=i$, then only color $3-i$ satisfies the condition ($\overrightarrow{{\rm WO}}$)  at $x$. Without loss of generality, assume that there is an arc $xy\in A(T)$ with $\phi(xy)=1$, which implies that the condition ($\overrightarrow{{\rm WO}}$) is satisfied by color $2$ at $x$.
Then there is a vertex $z$ such that $zx\in A(T)$ with $\phi(zx)=2$. We can define $\varphi$ as follows: $\varphi(zx)=1$, $\varphi(zy)=\{1,2\}$, $\varphi(xy)=2$ and $\varphi=\phi$ for other arcs.
It is easy to check that any vertex in $V(T)\setminus \{x,y,z\}$ is taken care of by the same color as in $\phi$. Note that $T$ is a tournament of odd order, then $d^+_T(x)$ and $d^-_T(x)$ have the same parity. Combining that the condition ($\overrightarrow{{\rm WO}}$) is satisfied by color $2$ at $x$, we have  both $d^+_{T'_2}(x)$ and $d^-_{T'_2}(x)$ are odd whereas $d^+_{T'_1}(x)$ and $d^-_{T'_1}(x)$ are even. Let $T_i$ be the spanning subdigraph of $T$ whose arc set is the color set $\varphi^{-1}(i)$ for $i\in [2]$. We finish the proof by the following.

Assume that the condition ($\overrightarrow{{\rm WO}}$) is satisfied by color $i$ at $z$ under $\phi$ for $i\in[2]$.
 Then we have $\phi(zy)=3-i$, $d^+_{T'_i}(z)$ and $d^-_{T'_i}(z)$ are odd whereas $d^+_{T'_{3-i}}(z)$ and $d^-_{T'_{3-i}}(z)$ are even. Combining the coloring $\varphi$, we have $d^+_{T_1}(x)=d^+_{T'_1}(x)-1$, $d^-_{T_1}(x)=d^-_{T_1}(x)+1$, $d^+_{T_i}(z)=d^+_{T'_i}(z)+2(2-i)$, $d^-_{T_i}(z)=d^-_{T'_i}(z)$ and $d^-_{T_{3-i}}(y)=d^-_{T'_{3-i}}(y)-(-1)^i$ are odd.

%
%
%

Hence, we are done.
\end{proof}

\noindent{\bf Acknowledgements.} The authors would like to thank Yongtang Shi for interesting discussions and the anonymous referees for their reports which helped to improve the presentation of the paper. This work was partially supported by the National Natural Science Foundation of China and the Natural Science Foundation of Tianjin (Nos. 20JCJQJC00090 and 20JCZDJC00840).

\end{document}